\documentclass[a4paper,centertags,oneside,12pt]{amsart}
\usepackage{mathrsfs}
\usepackage{appendix}
\usepackage{amssymb}
\usepackage{fancyhdr}
\usepackage{charter}
\usepackage{typearea}
\usepackage{pdfsync}
\usepackage{mathrsfs}
\usepackage[a4paper,top=3cm,bottom=3cm,left=3cm,right=3cm]{geometry}

\usepackage{enumitem}
\usepackage{color}
\setlist[1]{itemsep=5pt}
\newcommand{\comment}[1]{}

\makeatletter
      \def\@setcopyright{}
      \def\serieslogo@{}
      \makeatother

\newtheorem{theorem}{Theorem}[section]
\newtheorem{lemma}[theorem]{Lemma}
\newtheorem{corollary}[theorem]{Corollary}

\newtheorem{problem}[theorem]{Problem}
\newtheorem{definition}[theorem]{Definition}
\newtheorem{conjecture}[theorem]{Conjecture}

\newtheorem{remark}[theorem]{Remark}

\numberwithin{equation}{section}

\begin{document}
\title{On a holomorphic family of Stein manifolds with strongly pseudoconvex boundaries}
\author{Xiaoshan Li}
\address{School of Mathematics
and Statistics, Wuhan University, Hubei 430072, China}
\thanks{* Corresponding author.}
\thanks{Xiaoshan Li was  supported by  NSFC  No. 11871380, No. 11501422.}
\email{xiaoshanli@whu.edu.cn}
\author[Guicong Su]{Guicong Su*}
\address{School of Mathematics
and Statistics, Wuhan University, Hubei 430072, China}
\thanks{Guicong Su was  supported by  NSFC  No. 11671306.}
\email{suguicong@whu.edu.cn}

\begin{abstract}
We study the stable embedding problem for a CR family of  $3$-dimensional  strongly pseudoconvex CR manifolds with each fiber bounding a stein manifold.
\end{abstract}

\maketitle
\tableofcontents
\section{Introduction}
Let $X$ be  a compact
strongly pseudoconvex CR manifold. The question of whether or not $X$ admits a CR embedding into
a complex Euclidean space has attracted a lot attention.
This amounts to showing that the manifold has a sufficiently
rich collection of global CR functions.
It was shown by Boutet de Monvel \cite{BdM1:74b}
that the answer is affirmative if the dimension  of $X$ is at least five.
In contrast, if $X$ has dimension three,
$X$ may not be even locally embeddable, see \cite{Ku82,Ni74}.
Furthermore, there are examples \cite{Bur:79,G94,Ro65} which show that even
when the CR structure on $X$ is locally embeddable (for example, when it is real analytic),
it can happen that the global CR functions on $X$ fail to separate points of $X$.
It was shown in \cite{BE90} that, in a rather precise sense,
``generic'' perturbations of the standard CR structure on the three
sphere are nonembeddable.

On the other hand, if a compact three dimensional strongly pseudoconvex CR manifold admits a
transversal CR $S^1$-action, it was shown by Lempert \cite{Le92},
Epstein \cite{Ep92} and recently in \cite{HM14, HHL15} by using the Szeg\H{o} kernel,
that such CR manifolds can always be CR embedded into a complex Euclidean space.

In recent years, much progress has been made in understanding the
embedding question from a deformational point of view,
that is, for CR structures which lie in a small neighborhood
of a fixed embedded structure,
see e.\,g.\ \cite{BE90,Ep92,EH00,HLY06,C15,Le92,Le94, M16, W04}.
\begin{problem}[\cite{Le94}]
Suppose $f: (X, HX, J)\rightarrow\mathbb C^k$ is a CR embedding, and let
$(X, HX, J^\prime)$ be another CR manifold with $J^\prime$ close to $J$.
Assuming $(X, HX, J^\prime)$ is CR embeddable into some $\mathbb C^l$,
does it follow that it also admits a CR embedding
$f^\prime: (X, HX, J^\prime)\rightarrow\mathbb C^k$ with $f^\prime$ close to $f$?
\end{problem}

If this holds for $J^\prime$ close to $J$ we say that $f$ is a stable embedding.
We say that two tensors are close if they are close in the
$C^\infty$ topology on the appropriate space.
The problem of stable embedding was first studied by
Tanaka \cite{T75}, who showed that for a smooth family of compact strictly pseudoconvex
CR manifolds  of dimension at least five,
any embedding is stable provided the dimension of first Kohn-Rossi cohomology groups of
the fibers do not depend on the parameter.

However, for three dimensional case Catlin and Lempert \cite{CL92} constructed an example to show
that the unstable embedding exists. They constructed a family of unit circle bundle over a
fixed compact Riemann surface. The instability of CR embedding of the unit circle bundles
is a consequence of the instability of the very ample line bundles.
But if a strictly pseudoconvex CR manifold $X$ admits an embedding in
$\mathbb C^2$, Lempert \cite{Le94} showed that this embedding is stable. In general, Lempert proposed the following
\begin{conjecture}[\cite{Le94}]\label{2018-05-09-c1}
Let $(X, HX, J)$ be a three dimensional strongly pseudoconvex CR manifold.
If $(X, HX, J)$ is the boundary of a stein manifold, then any CR embedding $f: (X, HX, J)\rightarrow\mathbb C^k$ is stable.
\end{conjecture}

Huang, Luk and Yau \cite{HLY06} studied the stability of embedding  for a CR family of strongly pseudoconvex CR manifolds with the CR structures of the fibers CR depending on the parameters. In \cite{HLY06}, the dimension of each fiber has to be greater or equal to five. For a CR family of three dimensional CR manifolds, the problem of stability of embedding is still open. The CR dependence on the parameters for the CR families is crucial for the studies in the deformation theory of the complex structure of isolated singularities. Here, we refer the readers to \cite{BJ97, H08, HLY06} and the references therein.

In this paper, we will continue the program which was started in \cite{HLY06} on the stability of embedding problems for a CR family of strongly pseudoconvex CR manifolds. First, we recall the notations in \cite[Definition 1.1]{HLY06}.
\begin{definition}
Let $\Delta=\{t\in\mathbb C: |t|<1\}$ be the unit disk in the complex plane $\mathbb C$ and $\{X_t\}_{t\in\Delta}$ be a parameterized family of compact strongly pseudoconvex CR manifolds of real dimension $2n-1$. The family is said to be a CR family, or $X_{t_1}$ is said to be a CR deformation of $X_{t_2}$ for any $t_1, t_2\in\Delta$ if there is a strongly pseudoconvex CR manifold $\mathcal X$ and a $C^\infty$ CR map $\pi: \mathcal X\rightarrow\Delta$ such that (I) $\pi$ is a proper submersion; (II) for any $t\in\Delta$, $X_t=\pi^{-1}(t)$ and $X_t$ is a CR submanifold of $\mathcal X$.
\end{definition}
If $n\geq 3$, that is, the dimension of each fiber  is at least five, Huang-Luk-Yau \cite{HLY06} established the stability of embedding for a CR family of strongly pseudoconvex CR manifolds under the condition that the dimension of the first Kohn-Rossi cohomology of each fiber does not depend on the parameter. One key step in their work is the simultaneous filling of the CR family by a holomorphic family of convex-concave complex manifolds. The argument in \cite{HLY06} does not work on the case when the dimension of each fiber is three. An open problem was stated in \cite{H08}.
\begin{problem}[\cite{H08}]\label{2018-05-09-p1}
Let $\{X_t\}_{t\in\Delta}$ be a CR family of $3$-dimensional strongly pseudoconvex CR manifold. Suppose that the total space admits a normal stein filling. Suppose that $X_0$ is embedded into some $\mathbb C^N$. Under what conditions, can the nearby $X_t$ be CR embedded into the same $\mathbb C^N$?
\end{problem}
In order to study the conjecture \ref{2018-05-09-c1},  we assume that $X_0$ can be filled by a stein manifold $M_0$. Furthermore, we consider a special case of problem \ref{2018-05-09-p1}. We assume that the nearby $X_t$ can be simultaneously filled by a complex manifold $M_t$ with $X_t$ as its strongly pseudoconvex boundary and as a consequence we will have a holomorphic family of complex manifolds with strongly pseudoconvex boundaries. Here, when we say a complex manifold $\overline M$ with smooth boundary $X$ we mean that $\overline M$ has a cover by coordinate patches $\{U_\alpha\}$ with $C^\infty$ coordinates $\varphi_\alpha: U_\alpha\rightarrow\mathbb C^n$ such that $\varphi_\alpha: U_\alpha\cap M\rightarrow\mathbb C^n$ is holomorphic. Then $X$ inherits a natural CR structure from $M$. We may always assume that $M$ is a relatively compact open subset of some $C^\infty$ manifold $M'$ and $X$ is a smooth submanifold of $M'$.

\begin{definition}\label{20171212-Def1}
Let $\Delta=\{t\in\mathbb C: |t|<1\}$ be the unit disk in the complex plane $\mathbb C$ and $\{\overline{M_t}\}_{t\in\Delta}$ be a parameterized family of complex manifolds  with smooth boundaries. The family is said to be a holomorphic family if there is a complex manifold $\overline{\mathcal{M}}$ (with smooth boundary)  and a smooth map $\pi: \overline{\mathcal  M}\rightarrow \Delta$ such that (I) $\pi$ is a  proper submersion; (II) The restriction of $\pi$ on $\mathcal M$ is holomorphic; (III) for any $t\in\Delta$, $\overline{M_t}=\pi^{-1}(t)$ is a complex submanifold of $\overline{\mathcal M}$.
\end{definition}

In what follows, we denote by $(\overline{\mathcal M_{\varepsilon}}, \Delta_\varepsilon, \pi)$ a holomorphic family of complex manifolds with smooth boundaries where $\Delta_{\varepsilon}:=\{t\in\mathbb C: |t|<\varepsilon\}$. Suppose that ${M_0}$ is a stein manifold with  strongly pseudoconvex boundary.

We denote by $X_t=\partial M_t$ the boundary of $M_t$ for any $t\in\Delta_{\varepsilon}$ and set $\mathcal X_{\varepsilon}=\cup_{t\in\Delta_{\varepsilon}}X_t$. The boundary of $\mathcal M_\varepsilon$ has two pieces $Y_0$ and $Y_1$, where $Y_0=\mathcal X_\varepsilon$ and $Y_1=\pi^{-1}\{|t|=\varepsilon\}$. In what follows, we assume that $\mathcal M_\varepsilon$ is contained in a large differential manifold $\mathcal M'$ with $Y_0$ and $Y_1$ smooth submanifolds of $\mathcal M'$. Then we can state our main result.
\begin{theorem}\label{main theorem}
Let $(\overline{\mathcal M_{\varepsilon}}, \Delta_\varepsilon, \pi)$ be a holomorphic family of complex manifolds with each fiber of complex dimension $2$. Assume that $Y_0$ is strongly pseudoconvex with respect to $\mathcal M_\varepsilon$ and $M_0$ is a stein manifold.  Let $\overline{M_t}=\pi^{-1}(t)$ and $X_t=\partial M_t, ~\forall t\in\Delta_\varepsilon$.  If $X_0$ can be CR embedded into $\mathbb C^m$ for some $m$ by a CR map $F_0: X_0\rightarrow\mathbb C^m$, then there is a CR embedding $G: \mathcal X_\varepsilon\rightarrow \mathbb C^{m+1}$ when $\varepsilon$ is sufficiently small  such that $G|_{X_t}$ CR embeds $X_t$ into $\mathbb C^m\times\{t\}$ and $G|_{X_0}=(F_0, 0)$.
\end{theorem}

\section{$L^2$-method for the $\overline\partial$-equation on $\mathcal M_\varepsilon$}\label{dbar equation}
We now proceed to study the stability problem for a CR family of strongly pseudoconvex CR manifolds which bound a holomorphic family of complex manifolds. For this, we need study the $\overline\partial$-equation on $\mathcal M_\varepsilon$. However, the non-smooth boundary of $\mathcal M_\varepsilon$ makes a direct approach difficult. Thus, we need find a Hermitian metric on $\mathcal M_\varepsilon$ which can blow up $Y_0$ to infinity.

We now introduce a Hermitian metric $ds^2$ over $\mathcal M_\varepsilon$ such that the following properties hold: (a) First, $ds^2$ is smooth up to $\overline {\mathcal M_\varepsilon}\setminus Y_1$, (b) we can find a finite covering $\{U_\alpha\}_\alpha$ of $\overline{\mathcal M_\varepsilon}$ and coordinates $(z_\alpha, t)$ on each $U_\alpha$ such that $z_\alpha$ is smooth on $U_\alpha$ and $z_\alpha|_{\mathcal M_\varepsilon\cap U_\alpha}$ is holomorphic. For convenience, we will omit $\alpha$ in the notation $z_\alpha$. With respect to the coordinates $(z, t)$ with $z=(z_1, z_2)$,
$$ds^2=\sum_{j, k=0}^2 h_{j\overline k}(z,t)dz_j\otimes d\overline z_k+\frac{1}{(\varepsilon^2-|t|^2)^2}dt\otimes d\overline t,$$
where $z_0=t$ and $h_{j\overline k}\in C^\infty(U_\alpha\cap\overline{\mathcal M_\varepsilon})$ when $U_\alpha$ intersects with the boundary of $\mathcal M_\varepsilon$.  Write $\eta(t)=-\log{(\varepsilon^2-|t|^2)}$, $e^{2\eta(t)}=\frac{1}{(\varepsilon^2-|t|^2)^2}$. Then the volume form on $\mathcal M_\varepsilon$ with respect to $ds^2$ is given by $dv=h_0(z, t)e^{2\eta(t)}d_{Eucl}$, where $ h_0(z, t)\in C^\infty(U_\alpha\cap\overline {\mathcal M_\varepsilon})$ and $d_{Eucl}$ is the standard volume form on $\mathbb C^{3}$.

\begin{lemma}\label{2018-05-09-l1}
For sufficiently small $\varepsilon$, there exists a strictly plurisubharmonic function $\varphi$ on $\mathcal M_\varepsilon$ which is smooth up to the boundary $\partial \mathcal M_\varepsilon$. As a consequence, for  $t\in\Delta_\varepsilon$ each $M_t$ is a stein manifold with  strongly pseudoconvex boundary.
\end{lemma}
\begin{proof}
By the assumption of theorem \ref{main theorem}, $M_0$ is a stein manifold with a strongly pseudoconvex boundary. By theorem 4.1 in \cite{HN05}, there is a large stein manifold $Y$ which contain $M_0$ as an open subset. Then $M_0$ can be embedded to some $\mathbb C^N$ by a holomorphic map $F$ which is smooth up to the boundary. Then $\varphi_0=F^\ast(\sum_{j=1}^N |z_j|^2)$ is a strictly plurisubharmonic function on $M_0$ and $\varphi_0$ is smooth up to the boundary of $M_0$. Let $f: M_0\times \Delta_{\varepsilon}\rightarrow\mathcal M_\varepsilon$ be a diffeomorphism which is smooth up to the boundary satisfying $f|_{M_0\times\{0\}}=id.$ Let $p_r: M_0\times\Delta_\varepsilon\rightarrow M_0$ be the natural projection. Take  $\varphi=\varphi_0\circ p_r\circ f^{-1}+l|t\circ \pi|^2$ and $l$ is a positive number. Then $\varphi$ will be a strictly plurisubharmonic function on $\mathcal M_\varepsilon$ which is smooth up to the boundary  of $\mathcal M_\varepsilon$ when $l$ is sufficiently large and $\varepsilon$ is sufficiently small. Thus, each $M_t$ is a complex manifold without compact positive dimensional subvariety. Since each $M_t$ has a strongly pseudoconvex boundary, then by a result of Grauert \cite{G62} (also see \cite[Corollary, page 233]{AS70}) we have that each $M_t$ is a stein manifold.
\end{proof}
Lemma \ref{2018-05-09-l1} implies that the holomorphic family of complex manifolds in theorem
\ref{main theorem} is actually a holomorphic family of stein manifolds.

In what follows, we will fix a smooth real-valued function $r$ over $\overline{\mathcal M_\varepsilon}$ such that $r$ is a defining function of $Y_0$. For $\tau\gg1$, write $$\Lambda_{\tau, k}=-\tau\log{(\varepsilon^2-|t|^2)}+k\varphi$$ which is a strongly plurisubharmonic function on $\mathcal M_\varepsilon$ when $\tau$ is sufficient large. We will use $\Lambda_{\tau, k}$ as a weight function to solve the $\overline\partial$-equation on $\mathcal M_\varepsilon.$

We will work with the following orthonormal basis over $U_\alpha$:
\begin{equation}\label{basis}
\omega^0=\frac{b_0}{\varepsilon^2-|t|^2}dt, \omega^j=\sum_{k=1}^2 b^j_kdz_k+b_jdt, j=1, 2,
\end{equation}
where $b_0, b_j, b_k^j\in C^\infty(U_\alpha\cap\overline {\mathcal M_\varepsilon})$ and $b_0>0$. Let $\{L_j\}_{j=0}^2$ be the frame over $U_\alpha$ dual to $\{\omega^j\}_{j=0}^2$ in (\ref{basis}). Then
\begin{equation}
L_0=(\varepsilon^2-|t|^2)(a_0\frac{\partial}{\partial t}+\sum_{k=1}^2 a_k\frac{\partial}{\partial z_k}), L_j=\sum_{k=1}^2 a_j^k\frac{\partial}{\partial z_k}, j\neq 0.
\end{equation}
Here, $a_0, a_k, a_j^k\in C^\infty(U_\alpha\cap\overline{\mathcal M_\varepsilon})$.

Let $\Omega^{0, q}(\overline{\mathcal M_\varepsilon})$ be the space of $(0, q)$-forms on $\mathcal M_\varepsilon$ which are smooth up to the boundary of $\overline{\mathcal M_\varepsilon}$. Let $\Omega^{0, q}_c(\mathcal M_\varepsilon)$ be the subspace of $\Omega^{0, q}(\overline{\mathcal M_\varepsilon})$ with the elements having compact support in the interior of $\mathcal M_\varepsilon$. A smooth $(0, q)$-form $u\in\Omega^{0, q}(\overline{\mathcal M_\varepsilon})$ is said to have compact support along $t$-direction if it vanishes when $\varepsilon-|t|<c_u$ with $c_u$ a sufficiently small constant. We denote by $L^2_{(0, q)}(\mathcal M_\varepsilon, \Lambda_{\tau, k})$ the completion of $\Omega^{0, q}_c(\mathcal M_\varepsilon)$ under the $\Lambda_{\tau, k}$-weighted $L^2$-norm.

Consider the $\overline\partial$-operator as a maximally closed extended operator acting on a dense subspace of the $\Lambda_{\tau, k}$-weighted $L^2$-space of functions, $(0, 1)$-forms. Let $\overline\partial_{\Lambda_{\tau, k}}^\ast$ be its Hilbert adjoint from the space of $(0, 1)$-forms into the space of functions.
\begin{lemma}\label{density lemma}
For all $u\in{\rm Dom}(\overline\partial)\cap{\rm Dom}(\overline\partial_{\Lambda_{\tau, k}}^\ast)\cap L^2_{(0, 1)}(\mathcal M_\varepsilon, \Lambda_{\tau, k})$, there exists a sequence $\{u_m\}$ which belong to $\Omega^{0, 1}(\overline{\mathcal M_\varepsilon})$ and have compact supports along $t$-direction such that as $m\rightarrow\infty$,
\begin{equation}
\|u_m-u\|_{\Lambda_{\tau, k}}+\|\overline\partial u_m-\overline\partial u\|_{\Lambda_{\tau, k}}+\|\overline\partial_{\Lambda_{\tau, k}}^\ast u_m-\overline\partial^{\ast}_{\Lambda_{\tau, k}}u\|_{\Lambda_{\tau, k}}\rightarrow 0.
\end{equation}
\end{lemma}
\begin{proof}
For any $v\in\mathbb N$, choose $\eta_v\in C_0^\infty(\Delta_\varepsilon)$ such that $\eta_v(t)\equiv1$ when $|t|<\varepsilon-\frac{1}{v}$, $\eta_v(t)\equiv0$ when $|t|>\varepsilon-\frac{1}{2v}$. Then $|D\eta_v|\leq Cv$  and the point wise norm $|\overline\partial \eta_v|_{ds^2}\leq C_0$ for some constants $C, C_0$ independent of $v$. Hence, for any $u\in{\rm Dom}(\overline\partial)\cap{\rm Dom}(\overline\partial_{\Lambda_{\tau, k}}^\ast)\cap L^2_{(0, 1)}(\mathcal M_\varepsilon, \Lambda_{\tau, k})$, $\overline\partial(\eta_v u)\rightarrow\overline\partial u$ and $\overline\partial^{\ast}_{\Lambda_{\tau, k}}(\eta_v u)\rightarrow\overline\partial^{\ast}_{\Lambda_{\tau, k}} u$ in the respected $\Lambda_{\tau, k}$-weighted $L^2$-norms. Now applying the Friedrich-H\"ormander approximation theorem, we see that for any $u\in{\rm Dom}(\overline\partial)\cap{\rm Dom}(\overline\partial_{\Lambda_{\tau, k}}^\ast)\cap L^2_{(0, 1)}(\mathcal M_\varepsilon, \Lambda_{\tau, k})$, there is a sequence $\{u_m\}\subset{\rm Dom}(\overline\partial_{\Lambda_{\tau, k}}^\ast)$, each of which is compactly supported along $t$-direction and is smooth up to the boundary of $\overline{\mathcal M_\varepsilon}$ such that
\begin{equation}
u_m\rightarrow u, \overline\partial u_m\rightarrow\overline\partial u, \overline\partial^\ast_{\Lambda_{\tau, k}}u_m\rightarrow\overline\partial_{\Lambda_{\tau, k}}^\ast u
\end{equation}
in the respect $\Lambda_{\tau, k}$-weighted $L^2$-norms.
 \end{proof}

Let $u\in{\rm Dom}(\overline\partial^\ast_{\Lambda_{\tau, k}})\cap\Omega^{0, 1}(\overline {\mathcal M_\varepsilon})$. We assume that $u$ has compact support along $t$-direction. For $p\in\overline {\mathcal M_\varepsilon}$, there exists a neighborhood $U_\alpha$ of $p$ in $\mathcal M'$ and orthonormal basis $\{\omega^j\}_{j=0}^2$ given as in (\ref{basis}) over $U_\alpha$. By partition of unity, we assume that $u$ has compact support in some $U_\alpha\cap\overline {\mathcal M_\varepsilon}$. Write $u=\sum_{j=0}^2 u_j\overline\omega^j$. Since $u$ has compact support along $t$-direction and  satisfies $\overline\partial$-Neumann boundary condition, so we have $u|_{Y_1}=0$ and
 \begin{equation}
\sum_{j=0}^2 L_j(r)u_j=0~\text{along}~\overline {Y_0}\setminus Y_1.
\end{equation}

Write the volume form on $Y_0$ with respect to the Euclidean metric  on some coordinate chart as $ds$. Also, in what follows, we write $O(A)$ for a quantity such that $|O(A)|\leq C |A|$ with $C$ independent of the weight
function (namely, independent of $\tau, k$).
\begin{theorem}[Basic estimate]\label{basic estimate}
There exists $\tau, k$ sufficiently  large such that
for any $u\in{\rm Dom}(\overline\partial_{\Lambda_{\tau, k}}^\ast)\cap\Omega^{0, 1}(\overline{\mathcal{M_\varepsilon}})$ with $u$ having compact support along $t$-direction, we have
\begin{equation}\label{2018-4-15-e3}
\|\overline\partial u\|_{\Lambda_{\tau, k}}^2+\|\overline\partial_{\Lambda_{\tau, k}}^\ast u\|^2_{\Lambda_{\tau, k}} \gtrsim \|u\|^2_{\Lambda_{\tau, k}}+\sum_{\alpha, \beta=0}^2\int_{Y_0}r_{\alpha\overline\beta} u_\alpha\overline u_\beta e^{-{\Lambda_{\tau, k}}} h_0(z, t)e^{2\eta(t)}ds,
\end{equation}
where  $\{r_{\alpha\overline\beta}\}$ are given by $\partial\overline\partial r=\sum_{\alpha, \beta=0}^2 r_{\alpha\overline\beta}\omega^\alpha\wedge\overline \omega^\beta$.  Moreover, For $u\in{\rm Dom}(\overline\partial)\cap{\rm Dom}(\overline\partial_{\Lambda_{\tau, k}}^\ast)\cap L^2_{(0, 1)}({\mathcal{M_\varepsilon}}, \Lambda_{\tau, k})$, we have
\begin{equation}
\|\overline\partial u\|_{\Lambda_{\tau, k}}^2+\|\overline\partial_{\Lambda_{\tau, k}}^\ast u\|^2_{\Lambda_{\tau, k}}\gtrsim \|u\|^2_{\Lambda_{\tau, k}}.
\end{equation}
\end{theorem}
\begin{proof}
From Lemma \ref{density lemma}, we only need prove the first part of this theorem. By partition of unity, we assume that ${\rm supp}~ u\subset U_\alpha\cap\overline{\mathcal M_\varepsilon}$ for some neighborhood $U_\alpha$. Let $\{\omega^j\}$ be the orthonormal basis defined in (\ref{basis}) over $U_{\alpha}$.
\begin{equation}
\overline\partial u=\sum_{j<k}(\overline L_j u_k-\overline L_k u_j)\overline\omega^j\wedge\overline\omega^k+{\rm lower~order~terms}.
\end{equation}
Write $\||u\||^2_{\Lambda_{\tau, k}}=\sum_{\alpha, \beta=0}^2\|\overline L_\alpha u_\beta\|^2_{\Lambda_{\tau, k}}+\|u\|^2_{\Lambda_{\tau, k}}$. Then
\begin{equation}\label{norm1}
\begin{split}
\|\overline\partial u\|_{\Lambda_{\tau, k}}^2=&\sum\limits_{\alpha, \beta=0}^2\|\overline L_\alpha u_\beta\|_{\Lambda_{\tau, k}}^2-\sum\limits_{\alpha, \beta=0}^2\int_{U_p}\overline L_\alpha(u_\beta)\overline{\overline L_\beta(u_\alpha)}e^{-\Lambda_{\tau, k}} h_0e^{2\eta(t)}d_{Eucl}\\
&+O(\||u\||_{\Lambda_{\tau, k}}\cdot\|u\|_{\Lambda_{\tau, k}}).
\end{split}
\end{equation}
Let $v=\sum_{j=0}^2v_j\overline\omega^j$ and let $\chi\in C_0^\infty(U_\alpha\cap\mathcal M_\varepsilon)$. We have
\begin{equation}
\begin{split}
\int_{U_p}\langle\overline\partial\chi, v\rangle e^{-\Lambda_{\tau, k}}h_0e^{2\eta(t)}d_{Eucl}&=\sum_{j=0}^2\int_{U_p}\overline L_j(\chi)\overline v_j e^{-\Lambda_{\tau, k}}h_0e^{2\eta(t)}d_{Eucl}\\
&=\sum_{j=0}^2\int_{U_p}\chi\overline L_j^\ast(\overline v_j\tilde h)\tilde h^{-1}\tilde h d_{Eucl},
\end{split}
\end{equation}
where $\tilde h=e^{-\widetilde{\Lambda_{\tau, k}}}:=e^{-\Lambda_{\tau, k}}h_0e^{2\eta(t)}$ and $\overline L_j^\ast$ is the formal adjoint of $\overline L_j$ with respect to the Euclidean metric. Notice that $\overline L_j^\ast=-L_j+K_j$, with $K_j\in C^\infty(U_p)$ for all $j$. Since $u\in{\rm Dom}(\overline\partial^\ast_{\Lambda_{\tau, k}})$, we have
\begin{equation*}
\overline\partial^\ast_{\Lambda_{\tau,k}} u=-\sum_{j=0}^2\delta_j u_j+{\text{lower~order~terms}}
\end{equation*}
and
\begin{equation}\label{norm2}
\|\overline\partial_{\Lambda_{\tau, k}}^\ast u\|^2_{\Lambda_{\tau, k}}=\sum_{\alpha, \beta=0}^2(\delta_\alpha u_\alpha, \delta_\beta u_\beta)_{\Lambda_{\tau, k}}+O(\|\overline\partial^\ast_{\Lambda_{\tau, k}}u\|_{\Lambda_{\tau, k}}\cdot\|u\|_{\Lambda_{\tau, k}})+O(\|u\|_{\Lambda_{\tau, k}}^2),
\end{equation}
where $(\cdot, \cdot)_{\Lambda_{\tau, k}}$ is the weighted inner product with respect to $ds^2$ and the weight function $\Lambda_{\tau, k}$. Here, $\delta_j u_j=e^{\widetilde{\Lambda_{\tau, k}}}L_j(e^{-\widetilde{\Lambda_{\tau, k}}}u_j)$.
Combining (\ref{norm1}) and (\ref{norm2}), we have
\begin{equation}
\begin{split}
&\|\overline\partial u\|^2_{\Lambda_{\tau, k}}+\|\overline\partial^\ast_{\Lambda_{\tau, k}}u\|_{\Lambda_{\tau, k}}^2\\
=&\sum_{\alpha, \beta=0}^2\|\overline L_\alpha u_\beta\|_{\Lambda_{\tau, k}}^2+\sum_{\alpha, \beta=0}^2(\delta_\alpha u_\alpha, \delta_\beta u_\beta)_{\Lambda_{\tau, k}}-\sum_{\alpha, \beta=0}^2(\overline L_\alpha u_\beta, \overline L_\beta u_\alpha)_{\Lambda_{\tau, k}}\\
&+O(\|\overline\partial^\ast_{\Lambda_{\tau, k}}u\|_{\Lambda_{\tau, k}}\cdot\|u\|_{\Lambda_{\tau, k}})+O(\||u\||_{\Lambda_{\tau, k}}\cdot\|u\|_{\Lambda_{\tau, k}}).
\end{split}
\end{equation}
Since $u$ satisfies $\overline\partial$-Neumann boundary condition, then integrating by parts
\begin{equation}\label{inte-18-4-11}
\begin{split}
&(\delta_\alpha u_\alpha, \delta_\beta u_\beta)_{\Lambda_{\tau, k}}\\
&=-(\overline L_\beta\delta_\alpha u_\alpha, u_\beta)_{\Lambda_{\tau, k}}+O(\|\delta_\alpha u_\alpha\|_{\Lambda_{\tau, k}}\cdot\|u\|_{\Lambda_{\tau, k}})\\
&=((\delta_\alpha\overline L_\beta-\overline L_\beta\delta_\alpha)u_\alpha, u_\beta)_{\Lambda_{\tau, k}}-(\delta_\alpha\overline L_\beta u_\alpha, u_\beta)_{\Lambda_{\tau, k}}+O(\|\delta_\alpha u_\alpha\|_{\Lambda_{\tau, k}}\cdot\|u\|_{\Lambda_{\tau, k}})\\
&=((\delta_\alpha\overline L_\beta-\overline L_\beta\delta_\alpha)u_\alpha, u_\beta)_{\Lambda_{\tau, k}}+(\overline L_\beta u_\alpha, \overline L_\alpha u_\beta)_{\Lambda_{\tau, k}}-\int_{Y_0}L_\alpha(r)\overline L_\beta(u_\alpha)\overline u_\beta e^{-\widetilde{\Lambda_{\tau, k}}}ds\\
&~~~~~+O(\||u\||_{\Lambda_{\tau, k}}\cdot\|u\|_{\Lambda_{\tau, k}}).
\end{split}
\end{equation}
Since $u$ satisfies $\overline\partial$-Neumann boundary condition,  it follows that   $\sum_{\alpha=0}^2 u_\alpha L_\alpha (r)=0$ on $Y_0\setminus Y_1$. Making use of Morrey's trick we have on $\overline {Y_0}\setminus Y_1$,
\begin{equation*}
\sum_{\beta=0}^2\overline u_\beta\overline L_\beta(\sum_{\alpha=0}^2 u_\alpha L_\alpha(r))=0,
\end{equation*}
that is,
\begin{equation}\label{Morrey trick}
-\sum_{\alpha, \beta=0}^2 L_\alpha(r)\overline L_\beta(u_\alpha)\overline u_\beta=\sum_{\alpha, \beta=0}^2\overline L_\beta L_\alpha(r)u_\alpha\overline u_\beta.
\end{equation}
Substituting (\ref{Morrey trick}) to (\ref{inte-18-4-11}), we have
\begin{equation}\label{2018-4-12-e4}
\begin{split}
&\|\overline\partial u\|^2_{\Lambda_{\tau, k}}+\|\overline\partial^\ast_{\Lambda_{\tau, k}}u\|_{\Lambda_{\tau, k}}^2\\
&=\sum_{\alpha, \beta=0}^2\|\overline L_\alpha u_\beta\|_{\Lambda_{\tau, k}}^2+\sum_{\alpha, \beta=0}^2((\delta_\alpha\overline L_\beta-\overline L_\beta\delta_\alpha)u_\alpha, u_\beta)_{\Lambda_{\tau, k}}+\sum_{\alpha, \beta=0}^2\int_{Y_0}\overline L_\beta L_\alpha(r)u_\alpha\overline u_\beta e^{-\widetilde{\Lambda_{\tau, k}}}ds\\
&+O(\||u\||_{\Lambda_{\tau, k}}\cdot\|u\|_{\Lambda_{\tau, k}})+O(\|\overline\partial^\ast_{\Lambda_{\tau, k}}u\|_{\Lambda_{\tau, k}}\cdot\|u\|_{\Lambda_{\tau, k}}).
\end{split}
\end{equation}
By direct calculation,
\begin{equation}\label{2018-4-12-e2}
((\delta_\alpha\overline L_\beta-\overline L_\beta\delta_\alpha)u_\alpha, u_\beta)_{\Lambda_{\tau, k}}=((L_\alpha\overline L_\beta-\overline L_\beta L_\alpha)u_\alpha, u_\beta)_{\Lambda_{\tau, k}}+((\overline L_\beta L_\alpha\widetilde{\Lambda_{\tau, k}})u_\alpha, u_\beta)_{\Lambda_{\tau, k}}.
\end{equation}
Write $\partial\omega^{\overline\beta}=C^{\overline\beta}_{t\overline s}w^t\wedge w^{\overline s}$  and
$\overline\partial\omega^\beta=C^\beta_{t\overline s}\omega^t\wedge\omega^{\overline s}$, where both $C^{\overline\beta}_{t\overline s}$, $C^\beta_{t\overline s}\in C^\infty(U_p\cap\overline{\mathcal M_\varepsilon})$ and $C^\beta_{t\overline s}=-\overline{C^{\overline\beta}_{s\overline t}}.$
For any smooth function $g$, write $\partial\overline\partial g=\sum_{\alpha, \beta=0}^2 g_{\alpha\overline\beta}\omega^\alpha\wedge\omega^{\overline \beta}.$ By direct calculation,
\begin{equation}
\partial\overline\partial g=[L_\alpha\overline L_\beta g+(\overline L_tg)C^{\overline t}_{\alpha, \overline\beta}]\omega^\alpha\wedge\omega^{\overline \beta},
\end{equation}
\begin{equation}
\overline\partial\partial g=[-\overline L_\beta L_\alpha g+(L_tg)C^t_{\alpha\overline \beta}]\omega^\alpha\wedge \overline\omega^\beta.
\end{equation}
and
\begin{equation}\label{2018-4-12-no1}
g_{\alpha\overline\beta}=L_\alpha\overline L_\beta g+(\overline L_t g)C^{\overline t}_{\alpha \overline\beta}.
\end{equation}
The $\partial\overline\partial g+\overline\partial\partial g=0$ implies that
\begin{equation}\label{2018-4-12-e1}
L_\alpha\overline L_\beta g-\overline L_\beta L_\alpha g=(\overline L_t g)\overline{C^t_{\beta\overline\alpha}}-(L_tg)C^{t}_{\alpha\overline\beta}.
\end{equation}
Substituting (\ref{2018-4-12-e1}) to the first term on the right hand side of (\ref{2018-4-12-e2}), we have
\begin{equation}
\begin{split}
((L_\alpha\overline L_\beta-\overline L_\beta L_\alpha)u_\alpha, u_\beta)_{\Lambda_{\tau, k}}
=(\overline{C^r_{\beta\overline\alpha}}(\overline L_r u_\alpha)-C^t_{\alpha\overline\beta} L_tu_\alpha, u_\beta)_{\Lambda_{\tau, k}}.
\end{split}
\end{equation}
Integrating by parts,
\begin{equation}\label{2018-4-12-e3}
\begin{split}
&-(C^t_{\alpha\overline\beta} L_t(u_\alpha), u_\beta)_{\Lambda_{\tau, k}}\\
&=-\int_{Y_0}C^t_{\alpha\overline\beta}L_t(r)u_\alpha \overline u_{\beta}e^{-\widetilde{\Lambda_{\tau, k}}}d_{Eucl}-\int_{U_p}L_t(\widetilde{\Lambda_{\tau, k}})C^t_{\alpha\overline\beta}u_\alpha\overline u_\beta e^{-\widetilde{\Lambda_{\tau, k}}}d_{Eucl}\\
&+\int_{U_p}L_t(C^t_{\alpha\overline\beta})u_\alpha\overline u_\beta e^{-\widetilde{\Lambda_{\tau, k}}}d_{Eucl}+\int_{U_p}C^t_{\alpha\overline\beta}u_\alpha L_t(\overline u_\beta)e^{-\widetilde{\Lambda_{\tau, k}}}d_{Eucl}+O(\|u\|_{\Lambda_{\tau, k}}^2).
\end{split}
\end{equation}
Combining (\ref{2018-4-12-e2}), (\ref{2018-4-12-e1}), (\ref{2018-4-12-e3}), (\ref{2018-4-12-e4}) and using the notation (\ref{2018-4-12-no1}) we have
\begin{equation}
\begin{split}
&\|\overline\partial u\|_{\Lambda_{\tau, k}}^2+\|\overline\partial_{\Lambda_{\tau, k}}^\ast u\|_{\Lambda_{\tau, k}}^2\\
=&\sum_{\alpha, \beta=0}^2\|\overline L_\alpha u_\beta\|_{\Lambda_{\tau, k}}^2
+\sum_{\alpha, \beta=0}^2\int_{U_p}(\widetilde{\Lambda_{\tau, k}})_{\alpha\overline\beta}u_\alpha\overline u_\beta e^{-\widetilde{\Lambda_{\tau, k}}}d_{Eucl}
+\sum_{\alpha, \beta=0}^2\int_{Y_0}r_{\alpha\overline\beta} u_\alpha\overline u_\beta e^{-\widetilde{\Lambda_{\tau, k}}}ds\\
&+O(\|\overline\partial_{\Lambda_{\tau, k}}^\ast u\|_{\Lambda_{\tau, k}}\cdot\|u\|_{\Lambda_{\tau, k}})+O(\||u\||_{\Lambda_{\tau, k}}\cdot\|u\|_{\Lambda_{\tau, k}}).
\end{split}
\end{equation}
Recall $$\widetilde{\Lambda_{\tau, k}}=(-\tau+2)\log{(\varepsilon^2-|t|^2)}+k\varphi(z, t)-\log{h_0(z, t)}$$ and $\partial\overline\partial\widetilde{\Lambda_{\tau, k}}=\sum_{\alpha, \beta=0}^n(\widetilde{\Lambda_{\tau, k}})_{\alpha\overline\beta}\omega^\alpha\wedge\omega^{\overline\beta}$, $(-\log{(\varepsilon^2-|t|^2)})_{0\overline0}=1$, then we can choose $\tau, k$ sufficiently large such that
\begin{equation}
((\widetilde{\Lambda_{\tau, k}})_{\alpha\overline\beta})\geq k(\delta_{\alpha\beta}),
\end{equation}
where $(\delta_{\alpha\beta})$ is the identity matrix. Then by big-small constants argument, we have
\begin{equation}\label{basic estimate1}
\|\overline\partial u\|_{\Lambda_{\tau, k}}^2+\|\overline\partial_{\Lambda_{\tau, k}}^\ast u\|_{\Lambda_{\tau, k}}^2 \gtrsim \sum_{\alpha, \beta=0}^2\|\overline L_\alpha u_\beta\|_{\Lambda_{\tau, k}}^2 + k\|u\|^2_{\Lambda_{\tau, k}}+\sum_{\alpha, \beta=0}^2\int_{Y_0}r_{\alpha\overline\beta} u_\alpha\overline u_\beta e^{-\widetilde{\Lambda_{\tau, k}}}ds.
\end{equation}
Since $Y_0$ is strongly pseudoconvex, then the boundary term in (\ref{basic estimate1}) is positive and we have
\begin{equation}
\|\overline\partial u\|_{\Lambda_{\tau, k}}^2+\|\overline\partial_{\Lambda_{\tau, k}}^\ast u\|_{\Lambda_{\tau, k}}^2\gtrsim k\|u\|^2_{\Lambda_{\tau, k}}.
 \end{equation}
 Then by partition of unity and  Lemma \ref{density lemma} we get the conclusion of the second part of Theorem \ref{basic estimate}.
\end{proof}
The basic estimate in theorem \ref{basic estimate} implies that there is no obstruction for solving the $\overline\partial$-equation on $\mathcal M_\varepsilon$. That is, if we denote by $$H^{q}_{(2)}(\mathcal M_\varepsilon, \Lambda_{\tau, k}): =\frac{{\rm Ker}~\overline\partial: L^2_{(0, q)}(\mathcal M_\varepsilon, \Lambda_{\tau, k})\rightarrow L^2_{(0, q+1)}(\mathcal M_\varepsilon, \Lambda_{\tau, k})}{{\rm Im}~\overline\partial: L^2_{(0, q-1)}(\mathcal M_\varepsilon, \Lambda_{\tau, k})\rightarrow L^2_{(0, q)}(\mathcal M_\varepsilon, \Lambda_{\tau, k})}$$ the $L^2$ Dolbeault cohomology, then we have the following
\begin{corollary}\label{2018-07-08e1}
 For sufficiently large $\tau$ and $k$, when $q\geq 1$,
$H^q_{(2)}(\mathcal M_\varepsilon, \Lambda_{\tau, k})=0$.
\end{corollary}

\begin{remark}
The assumption that each fiber of the family $\pi: \overline {\mathcal M_\varepsilon}\rightarrow\Delta_\varepsilon$ is  a stein  manifold with strongly pseudoconvex boundary is crucial in this paper. Without this assumption there will be compact analytic variety in $\mathcal M_\varepsilon$ and we can not find  a strictly plurisubharmonic function as in lemma \ref{2018-05-09-l1} which play an important role as a weight function in establishing the $L^2$-estimate for the $\overline\partial$-operator on $\mathcal M_\varepsilon$. Without such weight function, the obstruction $H^q_{(2)}(\mathcal M_\varepsilon, \Lambda_{\tau, k}), q\geq 1$ of solving the $\overline\partial$-equation on $\mathcal M_\varepsilon$ may be an infinite dimensional space.
\end{remark}

Let $\omega$ be a $\overline\partial$-closed $\Lambda_{\tau, k}$-weighted $L^2$-integrable $(0, 1)$-form. Then by Corollary \ref{2018-07-08e1}, the $\overline\partial$-equation $\overline\partial u=\omega$ always has a unique solution $u\in L^2_{(0, 1)}(\mathcal M_\varepsilon, \Lambda_{\tau, k})$ with $u\perp {\rm Ker}(\overline\partial).$ Next, we will show that the $\overline\partial$-equation on $\mathcal M_\varepsilon$ has boundary regularity.
\begin{theorem}\label{boundary regaularity}
Let $\omega$ be a closed $\Lambda_{\tau, k}$-weighted $L^2$-integrable $(0, 1)$-form. Suppose $\omega\in C^\infty(\overline{\mathcal M_\varepsilon}\setminus Y_1)$. There exist a unique $u\in L^2(\mathcal M_\varepsilon, \Lambda_{\tau, k})\cap C^\infty(\overline{\mathcal M_\varepsilon}\setminus Y_1)$ such that $\overline\partial u=\omega$ with $u\perp{\rm Ker}(\overline\partial).$
\end{theorem}
\begin{proof}
Let $\overline\partial: L^2(\mathcal M_\varepsilon, \Lambda_{\tau, k})\rightarrow L^2_{(0, 1)}(\mathcal M_\varepsilon, \Lambda_{\tau, k})$.  Let $\overline\partial^\ast_{\Lambda_{\tau, k}}$ be its Hilbert adjoint. There exist a unique $u\in L^2(\mathcal M_\varepsilon, \Lambda_{\tau, k})$ such that $\overline\partial u=\omega$ with $u\perp{\rm Ker}(\overline\partial)$. By Theorem \ref{basic estimate}, both $\overline\partial$ and $\overline\partial_{\Lambda_{\tau, k}}^\ast$ have closed range in $L^2_{(0, 1)}(\mathcal M_\varepsilon, \Lambda_{\tau, k})$ and $L^2(\mathcal M_\varepsilon, \Lambda_{\tau, k})$ respectively. Then $u\perp{\rm Ker}(\overline\partial)$ implies that $u\perp {\rm Rang}(\overline\partial_{\Lambda_{\tau, k}}^\ast)^{\perp}$ and thus $u\in{\rm Rang}(\overline\partial_{\Lambda_{\tau, k}}^\ast)$. There exist a $\beta\in L^2_{(0, 1)}(\mathcal M_\varepsilon, \Lambda_{\tau, k})$ with $\beta\perp{\rm Ker}(\overline\partial_{\Lambda_{\tau, k}}^\ast)$ such that $\overline\partial_{\Lambda_{\tau, k}}^\ast\beta=u$. Hence $\overline\partial\beta=0$ and thus $(\overline\partial\overline\partial_{\Lambda_{\tau, k}}^\ast+\overline\partial_{\Lambda_{\tau, k}}^\ast\overline\partial)\beta=\omega.$
By the localized version of Kohn's subelliptic estimate near $Y_0$ (see \cite[Chapter 2]{FK}) and \cite[Proposition 3.1.1, 3.1.11]{FK} we conclude that $\beta\in C^\infty(\overline{\mathcal M_\varepsilon}\setminus Y_1)$. The proof of this conclusion  is just a minor modification of the argument in the proof of the main theorem in \cite [Main Theorem 2.1.7]{FK}. For the convenience of the readers, we include the necessary modification here. We will follow the notation of \cite[Chapter 2]{FK}. We now define $Q(\phi, \phi)=\|\overline\partial\phi\|^2_{\Lambda_{\tau, k}}+\|\overline\partial_{\Lambda_{\tau, k}}^\ast\phi\|_{\Lambda_{\tau, k}}^2+\|\phi\|_{\Lambda_{\tau, k}}^2$ for $\phi\in\mathcal D^{(0, 1)}$, where $\mathcal D^{(0, 1)}$ is now a space of $(0, 1)$-forms $\phi\in C^\infty(\overline{\mathcal M_\varepsilon}\setminus Y_1)\cap{\rm Dom}(\overline\partial_{\Lambda_{\tau, k}}^\ast)$ with
compact support in the $t$-direction. $Q^\delta$ is defined in the same way as in \cite[pp32]{FK} with $\{\rho_j\}$ a partition of unity subordinate to a finite covering $\{U_{p_j}\}$ of $\overline{\mathcal M_\varepsilon}$.  Let $\alpha$ be any $\Lambda_{\tau, k}$-weighted $L^2$-integrable $(0, 1)$-form. We assume that $\alpha\in C^\infty(\overline{\mathcal M_\varepsilon}\setminus Y_1)$ and $\alpha=F(\phi)$ for some $\phi\in \tilde{\mathcal D}^{(0, 1)}$ where $\tilde{\mathcal D}^{(0, 1)}$  is the completion of $\mathcal D^{(0, 1)}$ under the $\Lambda_{\tau, k}$-weighted $Q$-norm.
For any given point $p\in \overline {Y_0}\setminus Y_1$ and a neighborhood $U_p$ of $p$ in $\overline{\mathcal M_{\varepsilon}}$ with $U_p$ not intersecting with $Y_1$,  to see that the \cite[Main Theorem 2.1.7 (2)]{FK} holds over $U_p$, we only need explain why $\xi \phi^{\delta_l}$ converges to $\xi \phi$ as $\delta_l\rightarrow 0^{+}$ for any $\xi\in C_0^\infty(U_p\cap\overline{\mathcal M_{\varepsilon}})$. Let $\{\eta_v\}$ be the sequence defined in the proof of Lemma \ref{density lemma}.  Substituting $\eta_v\phi^{\delta_l}$ to the  (\ref{2018-4-15-e3}), we have
\begin{equation}
\|\overline\partial (\eta_v\phi^{\delta_l})\|_{\Lambda_{\tau, k}}^2+\|\overline\partial_{\Lambda_{\tau, k}}^\ast (\eta_v\phi^{\delta_l})\|^2_{\Lambda_{\tau, k}} \gtrsim \|\eta_v\phi^{\delta_l}\|^2_{\Lambda_{\tau, k}}+C_v\int_{Y_0}|\eta_v\phi^{\delta_l}|^2 e^{-{\Lambda_{\tau, k}}}h_0e^{2\eta(t)}ds
\end{equation}
where $C_v$ is a constant which does not depend on $\phi^{\delta_l}$.
Then as in \cite[pp45]{FK}, there is a constant $C_{v, s}$ such that
\begin{equation}\label{2018-4-15-e2}
\|\eta_v \phi^{\delta_l}\|_{s}\leq C_{v, s}~\text{uniformly~as}~\delta_l\rightarrow0
\end{equation}
with $s\geq 1, v\gg1.$ Then by a diagonal selecting process and making use of Rellich lemma, we can assume that $\phi^{\delta_l}$ converges to $\phi_0$ in the Sobolev $\|\cdot\|_s$-norm for each $s\geq 1, v\gg1$ over any compact subset of $\overline{\mathcal M_\varepsilon}\setminus Y_1$. Thus, by Sobolev embedding theorem $\phi_0\in C^\infty(\overline{\mathcal M_\varepsilon}\setminus Y_1)$. Notice that $Q(\phi^{\delta_l}, \phi^{\delta_l})\leq Q^{\delta_l}(\phi^{\delta_l}, \phi^{\delta_l})=(\alpha, \phi^{\delta_l})$, where $\alpha=F(\phi)=F^{\delta_l}(\phi^{\delta_l})$. By a big-small constant argument, it follows that
\begin{equation}\label{2018-4-15-e1}
\|\overline\partial\phi^{\delta_l}\|_{\Lambda_{\tau, k}}\lesssim \|\alpha\|_{\Lambda_{\tau, k}},
\|\overline\partial^\ast_{\Lambda_{\tau, k}}\phi^{\delta_l}\|_{\Lambda_{\tau, k}}\lesssim \|\alpha\|_{\Lambda_{\tau, k}}, \|\phi^{\delta_l}\|_{\Lambda_{\tau, k}}\lesssim \|\alpha\|_{\Lambda_{\tau, k}}.
\end{equation}
An immediate consequence of  (\ref{2018-4-15-e1}) is that $\phi_0\in L^2_{(0, 1)}(\mathcal M_\varepsilon, \Lambda_{\tau, k})\cap{\rm Dom}(\overline\partial^\ast_{\Lambda_{\tau, k}})\cap{\rm Dom}(\overline\partial)$. Obviously,  $\eta_v\phi_0\in\mathcal{D}^{(0, 1)}$ for all $v\gg1$ and the limit of $\{\eta_v \phi_0\}$ with respect to the $\Lambda_{\tau, k}$-weighted $Q$-norm is $\phi_0$, thus $\phi_0\in\tilde{D}^{(0, 1)}$. Next, for any $\psi\in\mathcal {D}^{(0, 1)}$,
\begin{equation}
Q(\phi, \psi)=(\alpha, \psi)=Q^{\delta_l}(\phi^{\delta_l}, \psi)=Q(\phi^{\delta_l}, \psi)+O(\delta_lC_{v, 1}\|\psi\|_1)
\end{equation}
for some $v\gg1$  where $C_{v, 1}$ is the constant given in (\ref{2018-4-15-e2}). Letting $\delta_l\rightarrow0$, we have
\begin{equation}
Q(\phi-\phi_0, \psi)=0, \forall \psi\in\mathcal {D}^{(0, 1)}.
\end{equation}
Thus, $\phi=\phi_0$. After \cite[Main theorem 2.1.7]{FK} is modified to $\mathcal M_{\varepsilon}$, \cite [Propositions 3.1.1, 3.1.11]{FK} need no change at all for deriving the smoothness of our $\beta$ to a small neighbourhood of $p$ in $\overline{\mathcal M_\varepsilon}\setminus Y_1$.
\end{proof}
\section{Simultaneous embedding of CR manifolds}
Using what we have established in section \ref{dbar equation}, we can now give a proof of the Theorem \ref{main theorem}. The key step is to prove the following global extension theorem.
\begin{theorem}\label{global extension}
Let $f$ be a holomorphic function on $M_0$ which is smooth up to the boundary $X_0$. Then $f$ admits a  holomorphic extension $\hat f$ on $\mathcal M_\varepsilon$ which is smooth over $\overline{\mathcal M_\varepsilon}\setminus Y_1$.
\end{theorem}
\begin{proof}
By a Lemma of Bell  \cite[Section4]{Be86}, we can find a finite open covering $\{U_\alpha\}$ of ${\mathcal M_\varepsilon}$ such that each $U_\alpha\subset \mathcal M$ is a  strongly pseudoconvex manifold  with connected smooth boundary and the intersection $\partial U_\alpha\cap\partial \mathcal M$ (if not empty) contains an open subset of the strongly pseudoconvex boundary of $\mathcal M$.  Suppose $\varepsilon$ is sufficiently small and we can assume that $U_\alpha\cap M_0\neq\emptyset$ for all $\alpha$. We denote by $f_\alpha$ the restriction of $f$ on $U_\alpha\cap M_0$ for each $\alpha$. Then each $f_\alpha$ is a holomorphic function which is smooth up to the boundary of $U_\alpha\cap M_0$. By \cite[Theorem 1]{A84}, $f_\alpha$ can be holomorphically extended to $U_\alpha$ and if we denote the extension by  $\tilde f_\alpha$ then  $\tilde f_\alpha$ is smooth up to the boundary of $U_\alpha$. Choose a partition of unity $\{\chi_\alpha\}$ subordinate to the covering $\{U_\alpha\}$. Put $\tilde f=\sum_\alpha\chi_\alpha\tilde f_\alpha$.  Choose a cut-off function $\chi(t)\in C_0^\infty(\Delta_\varepsilon)$ satisfying
$\chi\equiv1, \hbox{when} |t|\leq\frac{\varepsilon}{2}$. Set $\omega=\frac{1}{t}\overline\partial(\chi(t)\tilde f)$. Then by the same argument as in \cite[pp364]{HLY06} we have that $\omega\in\Omega^{0, 1}(\overline{\mathcal M_\varepsilon})$ and $\omega$ has compact support along $t$-direction. By Theorem \ref{boundary regaularity}, there exists $u\in C^\infty(\overline {\mathcal M_\varepsilon}\setminus Y_1)$ such that
$\overline\partial u=\omega.$
Thus, $\overline\partial(\chi(t)\tilde f-tu)=0$.  Write $\hat f=\chi(t)\tilde f-tu$. Then $\hat f$ is a holomorphic function on $\mathcal M_\varepsilon$ and  $\hat f\in C^\infty(\overline{\mathcal M_\varepsilon}\setminus Y_1)$, $\hat f|_{X_0}=f$.
\end{proof}
\subsection{Proof of Theorem \ref{main theorem}} \begin{proof}Let $F_0=(f_1, \cdots, f_m): X_0\rightarrow \mathbb C^m$ be a smooth CR embedding. Then by Bochner extension each $f_j$ can be holomorphically extended to $M_0$ which is still denoted by $f_j$ and $f_j$ is smooth up to the boundary $X_0$. By Theorem \ref{global extension}, each $f_j$ admits a holomorphic extension $\hat f_j$ to $\mathcal M_\varepsilon$. Moreover, each $\hat f_j\in C^\infty(\overline {\mathcal M_\varepsilon}\setminus Y_1)$ for $1\leq j\leq m$. Set $\hat G=(\hat f_1, \cdots, \hat f_m): \mathcal M_\varepsilon\rightarrow \mathbb C^m$. Then $\hat G$ is  a holomorphic map and $\hat G\in C^\infty(\overline {\mathcal M_\varepsilon}\setminus Y_1)$ with $\hat G|_{X_0}=F_0$. Set $G=(\hat G, \pi): \mathcal X_{\varepsilon}\rightarrow \mathbb C^{m+1}$ with $G(p)=(\hat G(p), \pi(p))~\forall~ p\in \mathcal X_{\varepsilon}$. Then $G$ is a CR embedding when $\varepsilon$ is sufficiently small and $G|_{X_t}: X_t\rightarrow \mathbb C^m\times\{t\}$ is a CR embedding. $G|_{X_0}=(F_0, 0)$.
\end{proof}
The arguments in this work give some partial result of Huang's problem \cite{H08} under the condition each fiber bounds a stein space without singularities. It does not work on  the cases when each fiber bounds a strongly pseudoconvex complex manifold which have compact analytic varieties in the interior. Motivated by Huang-Luk-Yau's work \cite{HLY06}, we state the following open problem:
\begin{problem}
Let $\{X_t\}_{t\in\Delta}$ be a CR family of $3$-dimensional strongly pseudoconvex CR manifolds. Suppose the total space $X=\cup_{t\in\Delta}X_t$ bounds a holomorphic family of complex manifolds $\{M_t\}_{t\in\Delta}$. Suppose that
$$\dim H_{(2)}^1(M_t)\equiv\hbox{constant}>0$$
where $H^{1}_{(2)}(M_t)=\frac{{\rm Ker}\overline\partial: L^2_{(0, 1)}(M_t)\rightarrow L^2_{(0, 2)}(M_t)}{{\rm Im}\overline\partial: L^2(M_t)\rightarrow L^2_{(0, 1)}(M_t)}.$
Assume that $X_0$ can be CR embedded into some $\mathbb C^N$. Can the nearby $X_t$ be CR embedded into the same $\mathbb C^N$ with the embedding maps CR depending on the parameters ?
\end{problem}
\begin{center}
{\bf Acknowledgement}
\end{center}

The  authors thank Professor Xiaojun Huang for introducing this problem to them and many helpful discussions during the preparation of this work. The first named author also thank Professor Huang for the constant encouragement and support in mathematics during these years.

\end{document}